                        \documentclass[12pt,twoside,a4paper]{amsart}

\usepackage{amssymb} 
\usepackage{amsmath}
\usepackage{tikz} 
\usepackage{tikz-cd}
\usepackage{xspace} 
\usepackage{enumerate}
\usepackage{microtype}

\newcommand{\fC}{\mathfrak{C}}

\newcommand{\bQ}{\mathbb{Q}}

\newcommand{\bA}{\mathbf{A}} 
\newcommand{\bB}{\mathbf{B}}
\newcommand{\bU}{\mathbf{U}}

\newcommand{\bN}{\mathbb{N}}
\newcommand{\bT}{\mathbf{T}}
\newcommand{\bC}{\mathbf{C}}
\newcommand{\bD}{\mathbf{D}}

\newcommand{\cC}{\mathcal{C}}
\newcommand{\cU}{\mathcal{U}}

\newcommand{\cK}{\mathcal{K}}

\newcommand{\restr}{\mathord{\upharpoonright}}

\newcommand{\ta}{{\bar{a}}}
\newcommand{\tb}{{\bar{b}}}
\newcommand{\tc}{{\bar{c}}}

\newcommand{\uSigma}{\underline{\Sigma}}

\newcommand{\JEP}{\ensuremath{\operatorname{JEP}}\xspace}
\newcommand{\AP}{\ensuremath{\operatorname{AP}}\xspace}
\newcommand{\HP}{\ensuremath{\operatorname{HP}}\xspace}
\newcommand{\HAP}{\ensuremath{\operatorname{HAP}}\xspace}
\newcommand{\AEP}{\ensuremath{\operatorname{AEP}}\xspace}

\newcommand{\Met}{\operatorname{Met}}
\newcommand{\cMet}{\operatorname{cMet}}

\newcommand{\Age}{\operatorname{Age}}

\newcommand{\Emb}{\operatorname{Emb}}
\newcommand{\End}{\operatorname{End}}
\newcommand{\Aut}{\operatorname{Aut}}

\newcommand{\ar}{\operatorname{ar}}

\newcommand{\Fraisse}{Fra\"\i{}ss\'e}

\newcommand{\injto}{\hookrightarrow}

\newcommand{\monoid}{\mathfrak}
\newcommand{\metsp}{\mathbb}

\theoremstyle{plain} \newtheorem{theorem}{Theorem}[section]
\newtheorem{proposition}[theorem]{Proposition}
\newtheorem{lemma}[theorem]{Lemma}
\newtheorem{corollary}[theorem]{Corollary} 
\theoremstyle{definition}
\newtheorem{definition}[theorem]{Definition}
\newtheorem*{problem}{Problem} 
\newtheorem{example}[theorem]{Example}

\theoremstyle{remark}
\newtheorem*{remark}{Remark}

\author[Ch.\,Pech]{Christian Pech}

\address{Institute of Algebra\\ TU Dresden}
\email{cpech@freenet.de}
\urladdr{http://www.math.tu-dresden.de/~pech}

\author[M.\,Pech]{Maja Pech} \address{Department of
  Mathematics\\University of Novi Sad} \curraddr{Institute of
  Algebra\\ TU Dresden}
\email{maja@dmi.uns.ac.rs, Maja.Pech@tu-dresden.de}
\urladdr{http://sites.dmi.rs/personal/pechm/}

\title[On  automatic homeomorphicity\dots]{On automatic homeomorphicity for transformation monoids}

\begin{document}

\begin{abstract}
    Transformation monoids carry a canonical topology --- the topology of point-wise convergence. A closed transformation monoid $\monoid{M}$ is said to have automatic homeomorphicity with respect to a class  $\cK$ of structures, if every monoid-isomorphism of $\monoid{M}$ to the endomorphism monoid of  a member of $\cK$ is automatically a homeomorphism. In this paper we show automatic homeomorphicity-properties for the monoid of non-decreasing functions on the rationals, the monoid of non-expansive functions on the Urysohn space and the endomorphism-monoid of the countable universal homogeneous poset. 
\end{abstract}

\maketitle

A major question in mathematics is to what extent the symmetries of a structure determine its properties (algebraic, geometric,\dots).  Of course the answer to this question depends strongly on the decision, what we consider to be a structure. For geometries this is essentially Felix Klein's Erlangen Program. However, if we take a much broader point of view and consider model theoretic structures, then the answer is that the automorphism group in general says little about the properties of a structure. Indeed, in some sense ``most'' structures have no nontrivial symmetries. 

This situation changes if we restrict the class of structures in question. For instance, by the Ryll-Nardzewski Theorem, a countable structure is $\omega$-categorical (i.e., determined among other countable structures up to isomorphism by its elementary theory) if and only if its automorphism group is oligomorphic (i.e., it has finitely many $k$-orbits for every $k\in\bN\setminus\{0\}$.). More or less a direct consequence of this theorem is that if $\bA$ is a countable $\omega$-categorical structure and if $\bB$ is a countable structure (possibly of different type than $\bA$), then  $\Aut(\bA)$ and $\Aut(\bB)$ are isomorphic as permutation groups if and only if $\bA$ and $\bB$ are first-order interdefinable. 

Every permutation  group can be endowed with a natural topology --- the topology of pointwise convergence---under which the  group operations are continuous. It was shown by Coquand, Ahlbrandt and Ziegler (cf. \cite{AhlZie86}) that two countable $\omega$-categorical structures are first order bi-interpretable if and only if their automorphism groups are isomorphic as topological groups. 

In \cite{DixNeuTho86} it was shown by Dixon, Neumann and Thomas that in the full symmetric  group $\monoid{S}_\bN$ on $\bN$ every subgroup of index less that $2^{\aleph_0}$ is open. In general, a countable structure $\bA$ is said to have the \emph{small index property}  if every subgroup of index less than $2^{\aleph_0}$ in $\Aut(\bA)$ is open. Thus, the above mentioned result says that the countable structure over the empty signature has the small index property. 

The small index property has strong consequences. Whenever $\bA$ and $\bB$ are countable structures, such that $\bA$ has the small index property, then every group-isomorphism from  $\Aut(\bA)$ to $\Aut(\bB)$ is continuous. By a result by Lascar \cite[Corollary 2.8]{Las91}, every continuous isomorphism between the automorphism groups of countable structures is already a homeomorphism. It follows that whenever $\bA$ and $\bB$ are countable structures and $\bA$ has the small index property then every group-isomorphism between $\Aut(\bA)$ and $\Aut(\bB)$ is a homeomorphism. Summing up, a countable structure with the small index property is determined among all other countable structures by its automorphism group (considered as abstract group) up to first order bi-interpretability. 

This observation has been spurring  the interest into  structures with the small index property.  A few corner-stones in the research about the small index property include \cite{DixNeuTho86,Her98,HerLas00,HodHodLasShe93,Hru92,KecRos07,Sol05,Tru89}. 

Another, rather different, approach to the reconstruction of $\omega$-cate\-gorical structures is due to Rubin \cite{Rub94} --- based on (weak) $\forall\exists$-interpre\-tations. It would go too far to describe this method at this place. However, if a countable $\omega$-categorical structure $\bA$ has a weak $\forall\exists$-interpretation and if $\bB$ is another $\omega$-categorical structure, then every isomorphism between the automorphism groups of $\bA$ and $\bB$ is a homeomorphism (cf. also \cite{BarPhD,Bar07}). Thus, a countable, $\omega$-categorical structure with a weak $\forall\exists$-interpretation  is determined among the $\omega$-categorical structures by its automorphism group (considered as an abstract group) up to first order bi-interpretability. 

The automorphism groups of first order structures have been the topic of intensive research. Much less is known about their endomorphism monoids. This situation is slowly changing as is witnessed by the papers \cite{BonDelDol10,DelDol04,Dol12,Dol14,Kub13,LocTru12,MalMitRus09,PecPec13a}, that deal with such diverse topics like the Bergman property, cofinality and strong cofinality, universality, idempotents, generic elements, and ideals of the endomorphism monoids of countable homogeneous structures.  

In this paper we will study the question, how much information about a relational structure can be recovered from its endomorphism monoid (considered as an abstract monoid). In particular, inspired by the group case and by a recent paper by Bodirsky, Pinsker and Pongr\'acz \cite{BodPinPon13}, we study when the endomorphism monoid of a relational structure $\bA$ has the property that every monoid isomorphism to the endomorphism monoid of a structure $\bB$ from a given class $\cK$ is automatically already a homeomorphism (here, the topology in question is always the canonical topology of pointwise convergence). Using the terminology of \cite{BodPinPon13}, this means, that $\End(\bA)$ has \emph{automatic homeomorphicity with respect to $\cK$}. 

In \cite{BodPinPon13} automatic homeomorphicity was shown for 
\begin{itemize}
	\item the monoid of injective functions on $\bN$,
	\item the full transformation monoid on $\bN$,
	\item the monoid of homomorphic self-embeddings of the Rado graph, and 
	\item the endomorphism monoid of the Rado graph,
	\item the monoid of homomorphic self-embeddings of the countable universal homogeneous digraph.
\end{itemize}
Moreover, automatic homeomorphicity with respect to the class of countable $\omega$-categorical structures was shown for 
\begin{itemize}
    \item the endomorphism monoid of the countable universal homogeneous tournament,
    \item the monoid of homomorphic self-embeddings of the countable universal homogeneous $k$-uniform hypergraph.
\end{itemize}
Recently we learnt that Truss, Vergas-Garc\'{\i}a \cite{TruVarPC} and Hyde \cite{HydPC} showed independently the automatic homeomorphicity for the endomorphism monoid of $(\bQ,<)$. 

We will extend this list by proving that the endomorphism monoid of the rationals $(\bQ,\le)$, the endomorphism monoid of the countable universal homogeneous poset  $(\mathbb{P},\le)$, and the monoid of non-expansive self-maps of the rational Urysohn-space all have automatic homeomorphicity with respect to the class of countable structures whose endomorphism  monoids have  just finitely many weak orbits (in the sense of \cite{Ste10}), cf.\ Theorem~\ref{reconQ}, Theorem~\ref{reconP}, and Theorem~\ref{reconU}.  

\section{Preliminaries}

\subsection{Transformation monoids}
For a set $A$, the set of all function from $A$ to $A$, equipped with composition of functions, forms a monoid $\monoid{T}_A$. The submonoids of $\monoid{T}_A$ are called \emph{transformation monoids on $A$}. The monoid $\monoid{T}_A$ is also called the \emph{full transformation monoid on $A$}. 

The submonoid of $\monoid{T}_A$ that consists of all permutations on $A$ is called the \emph{full symmetric group} on $A$. It will be denoted by $\monoid{S}_A$.

If we equip $A$ with the discrete topology, then $\monoid{T}_A$ is a product space of $A$. Thus, it is canonically equipped with the Tychonoff topology (a.k.a. the topology of pointwise convergence). For every $h\in \monoid{T}_A$, and for every finite subset $M$ of $A$, we consider the set 
\[\Phi_{h,M}:=\{f\in \monoid{T}_A\mid f\restr_M=h\restr_M\}.\]
Then 
\[\{\Phi_{h,M}\mid h\in \monoid{T}_A,\,M\subseteq A\text{ finite}\}.\]
forms a basis of the Tychonoff topology on $\monoid{T}_A$. Moreover, every transformation monoid on $A$ is canonically equipped with the corresponding subspace topology.
 
If $A$ is countably infinite, then the topology on $\monoid{T}_A$ is metrizable by an ultrametric. In particular, if $\ta=(a_i)_{i\in\omega}$ is an enumeration of $A$, then we consider the function
\begin{gather*}
  D_{\ta}\colon \monoid{T}_A\times \monoid{T}_A\to\omega^+\\ (f,g)\mapsto 
  \begin{cases} \min\{i\in\omega\mid f(a_i)\neq g(a_i)\} & f\neq g\\
    \omega & f=g.
  \end{cases}
\end{gather*}
Finally, the mentioned ultrametric on $\monoid{T}_A$ is given by
\[
d_\ta(f,g):= \begin{cases} 2^{-D_\ta(f,g)} & f\neq g\\
  0 & f=g,
\end{cases}
\]
for all $f,g\in \monoid{T}_A$.
     
It is easy to see that for every enumeration $\ta$ of $A$ and for all $f,g,h\in \monoid{T}_A$ we have $d_\ta(f,g)\le d_\ta(h\circ f, h\circ g)$, and that equality holds if $h$ is injective. In other words,  the metric $d_\ta$ is \emph{left-$\monoid{T}_A$-subinvariant}. 

In the following, whenever we deal with a transformation monoid $\monoid{M}\le \monoid{T}_A$, we implicitly consider it to be equipped with the topology of pointwise convergence. 

\subsection{Relational structures}
A \emph{relational signature} is a pair $\underline\Sigma=(\Sigma,\ar)$ where $\Sigma$ is a set of \emph{relational symbols} and $\ar\colon \Sigma\to\bN\setminus\{0\}$ assigns to each  relational symbol its \emph{arity}. With $\Sigma^{(n)}$ we will denote the set of all $n$-ary relational symbols in $\Sigma$.

A \emph{$\uSigma$-structure} $\bA$ is a pair $(A,(\varrho^\bA)_{\varrho\in\Sigma})$, such that $A$ is a set, and such that for each $\varrho\in\Sigma$ we have that $\varrho^\bA$ is a relation of arity $\ar(\varrho)$ on $A$.  The set $A$ will be called the \emph{carrier} of $\bA$ and the relations $\varrho^\bA$ will be called the \emph{basic relations} of $\bA$. If the signature $\uSigma$ is of no importance, we will speak only about relational structures. If not said otherwise the carrier of a $\uSigma$-structure $\bA$ will always be denoted by $A$ and the basic relations of $\bA$ will be denoted by $\varrho^\bA$ for each $\varrho\in\Sigma$.

Let $\bA$ and $\bB$ be $\uSigma$-structures. A function $h\colon A\to B$ is called a \emph{homomorphism} if for all $n\in\bN\setminus\{0\}$, for all $\varrho\in\Sigma^{(n)}$ and for all $\ta=(a_1,\dots,a_n)\in \varrho^\bA$ we have that $h(\ta):=(h(a_1),\dots,h(a_n))\in\varrho^\bB$. A function $h\colon A\to B$ is called \emph{embedding} if $h$ is injective and if for all $n\in\bN\setminus\{0\}$, for all $\varrho\in\Sigma^{(n)}$ and for all $\ta\in A^n$ we have 
\[ \ta\in\varrho^\bA \iff h(\ta)\in\varrho^\bB.\]
Surjective embeddings are called \emph{isomorphisms}. As usual, isomorphisms of a relational structure $\bA$ onto itself are called \emph{automorphisms}, and homomorphisms of $\bA$ to itself are called \emph{endomorphisms}. Moreover, embeddings of $\bA$ into itself will be called \emph{selfembeddings} of $\bA$. The set of all automorphisms, endomorphisms, and selfembeddings of $\bA$ will be denoted by $\Aut(\bA)$, $\End(\bA)$, and $\Emb(\bA)$, respectively. Clearly, $\Aut(\bA)$ is a permutation group, and $\End(\bA)$ and $\Emb(\bA)$ are transformation monoids.

Another word about notation: Whenever we write $h\colon \bA\to\bB$, we mean that $h$ is a homomorphism from $\bA$ to $\bB$. Moreover, with $h\colon \bA\injto\bB$ we denote the fact that $h$ is an embedding from $\bA$ into $\bB$. 

\begin{example}
    Consider the relational signature $\uSigma_M$ that contains for every $r\in\bQ^+\cup\{0\}$ a binary relational symbol $\varrho_r$. Then to every metric space $(A,d)$ we may associate a $\Sigma_M$-structures $\bA$ by defining 
\[
    \varrho_r^\bA :=\{(x,y)\in A^2\mid d(x,y)\le r\},
\]
for every $r\in\bQ^+\cup\{0\}$.  The metric $d$ can be reconstructed from $\bA$ by
\[
    d(x,y)=\inf\{r\in\bQ^+\cup\{0\}\mid (x,y)\in\varrho_r^\bA\}.
\]
To make this correspondence functorial, the proper choice  of morphisms between metric spaces are the non-expansive maps. Recall that a function $f\colon (A,d_A)\to (B,d_B)$ is called \emph{non-expansive} if for all $x,y\in A$ we have
\[ d_B(f(x),f(y))\le d_A(x,y).\]
With this definition of morphisms between metric spaces, the assignment $R\colon (A,d)\mapsto\bA$, $R\colon f\mapsto f$ is a full embedding into the category $\fC_{\uSigma_M}$ of all $\uSigma_M$-structures with homomorphisms as morphisms. Therefore, in the following we will identify metric spaces with their relational counter-parts. 
\end{example}

\subsection{Homogeneous relational structures}
Following \Fraisse, for every $\uSigma$-structure $\bA$, its age is the class of  of all finite $\uSigma$-structures that are embeddable into $\bA$. It will be denoted by $\Age(\bA)$. A $\uSigma$-structure $\bB$ is called \emph{younger} than $\bA$ if $\Age(\bB)\subseteq \Age(\bA)$. By $\overline{\Age(\bA)}$ we will denote the class of all countable $\uSigma$-structures younger than $\bA$. 

\begin{definition}
    A countable $\uSigma$-structure $\bA$. is called \emph{universal} if every structure from $\overline{\Age(\bA)}$ can be embedded into $\bA$. It is called \emph{homogeneous} if for every $\bB\in\Age(\bA)$ and for all embeddings $\iota_1,\iota_2\colon \bB\injto\bA$ there exists an automorphism $h$ of $\bA$ such that $\iota_2=h\circ \iota_1$. 
\end{definition}
\begin{remark}
    Our definition of homogeneity is equivalent to the more usual definition that every isomorphism between finite substructures of $\bA$ extends to an automorphism. Indeed, $\iota_1$ and $\iota_2$ mark two isomorphic copies of $\bB$ in $\bA$, and at the same time define an isomorphism between these two finite substructures given by $g\colon \iota_1(\bB)\to\iota_2(\bB)\colon  x\mapsto \iota_2(\iota_1^{-1}(x))$. Finally the postulated automorphism $h$ extends $g$. On the other hand, every isomorphism $g$ between finite substructures $\bB_1$ and $\bB_2$  of $\bA$ defines two embeddings $\iota_1\colon \bB_1\injto\bA$ and $\iota_2\colon \bB_2\injto\bA$, where $\iota_1$ is the identical embedding and $\iota_2=g\circ\iota_1$. Then every automorphism $h$ of $\bA$ that extends $g$ will satisfy $\iota_2=h\circ\iota_1$.     
\end{remark}

\begin{definition}
    Let $\cC$ be a class of $\uSigma$-structures. We say that $\cC$ has the 
    \begin{description}
		\item[hereditary property (\HP)]if whenever $\bA\in\cC$ and $\bB$ is a $\uSigma$-structure embeddable into $\bA$, then also $\bB\in\cC$,
  		\item[joint embedding property (\JEP)] if for all  $\bA,\bB\in\cC$ there exists a $\bC\in\cC$ and embeddings $f\colon \bA\injto \bC$ and $g\colon \bB\injto \bC$,
      	\item[amalgamation property (\AP)] if for all $\bA$, $\bB$, $\bC$  from $\cC$ and for all embeddings $f\colon \bA\injto \bB$, $g\colon \bA\injto \bC$, there exists $\bD\in\cC$  and embeddings $\hat{f}\colon \bC\injto \bD$, $\hat{g}\colon \bB\injto \bD$ such that the following  diagram commutes:
  \[
  \begin{tikzcd}
    \bC \rar[hook,dashed]{\hat{f}}& \bD\\
    \bA  \uar[hook]{g}\rar[hook]{f}&  \bB\uar[hook,dashed]{\hat{g}}.
  \end{tikzcd}
  \]
  \end{description}
\end{definition}
Let us recall the well-known characterization of ages of countable structures and, in particular,  of countable   homogeneous structures, by Roland \Fraisse:
\begin{theorem}[\Fraisse{} (\cite{Fra53})]
    Let $\cC$ be a class of finite $\uSigma$-structures. Then $\cC$ is equal to the age of a countable structure if and only if 
it has up to isomorphism just countably many members, and  it has the \HP and the \JEP. 
     Moreover, $\cC$ is equal to the age of a countable homogeneous structure if and only if it has in addition the \AP. Finally, any two countable homogeneous $\uSigma$-structures with the same age are isomorphic. 
\end{theorem}

\begin{definition}
    A class of finite $\uSigma$-structures is called an \emph{age} if it has the \HP, the \JEP, and if it contains up to isomorphism just countably many structures. An age is called a \emph{\Fraisse-class} if it has the \AP. A countable homogeneous $\uSigma$-structure $\bU$ is called the \emph{\Fraisse-limit} of $\Age(\bU)$. 
\end{definition}

\begin{example}\label{exFraisse}
    Some examples of \Fraisse-classes include:
    \begin{itemize}
        \item the class of finite simple graphs,
        \item the class of finite posets (strictly or non-strictly ordered),
        \item the class of finite linear orders (strictly or non-strictly ordered)
        \item the class of finite metric spaces with rational distances,
        \item the class of finite metric spaces with rational distances $\le 1$, 
        \item the class of finite tournaments.
    \end{itemize}
    The corresponding \Fraisse-limits are the Rado graph (aka.\ the countable random graph), the countable generic poset, the rationals, the rational Urysohn space, the rational Urysohn sphere, and the random tournament, respectively.  
\end{example}

\subsection{Homomorphism-homogeneous relational structures}
In \cite{CamNes06}, Cameron and Ne\v{s}et\v{r}il introduced several variants of the notion of homogeneity. One of these variations is homomorphism-homogeneity: 
\begin{definition}
    A countable $\uSigma$-structure $\bU$ is called  \emph{homomorphism-homogeneous} if for every $\bA\in\Age(\bU)$, for all embeddings $\iota\colon \bA\injto\bU$, and for all homomorphisms $h\colon \bA\to\bU$  there exists an endomorphism $\hat{h}$ of $\bU$ such that $h=\hat{h}\circ\iota$.  
\end{definition}
\begin{remark}
    This definition of homomorphism homogeneity slightly differs from the original given definition in \cite{CamNes06}. However, the equivalence of our definition to the original one is obvious. 
\end{remark}

The connection between the notions of homogeneity and homomor\-phism-homogeneity was created by Dolinka (cf.\ \cite[Proposition 3.8]{Dol14}):

\begin{definition} Let $\cC$ be a class of $\uSigma$-structures. We say that $\cC$ has the \emph{homo-amalgamation property} (\HAP) if for all $\bA, \bB\in\cC$, $g\colon \bA\injto\bB$, $\bT_1\in\cC$,
  $a\colon \bA\to \bT_1$ there exist $\bT_2\in\cC$, $b\colon \bB\to \bT_2$,
  $h\colon \bT_1\injto\bT_2$ such that the following diagram commutes:
  \[
  \begin{tikzcd}
    \bB \rar[dashed]{b}& \bT_2\\
    \bA\rar{a}\uar[swap,hook]{g} & \bT_1\uar[hook,dashed]{h} 
  \end{tikzcd}
  \]
\end{definition}

\begin{proposition}[{\cite[Proposition 3.8]{Dol14}}]\label{homhap}
    Let $\bU$ be a countable homogeneous structures. Then $\bU$ is homomorphism-homogeneous if and only if its age has the \HAP.
\end{proposition}
\begin{example}\label{exHAP}
Given the rather extensive literature on the classification of homomorphism-homogeneous structures, Proposition~\ref{homhap}  is a convenient tool for showing that the age of a given homogeneous structure has the \HAP:
    \begin{itemize}
        \item By \cite[Proposition 2.1]{CamNes06} the Rado graph is homomorphism-ho\-mogeneous. Thus, the class of finite graphs has the \HAP.
        \item  By \cite[Proposition 25]{CamLoc10} and \cite[Theorem 4.5]{Mas07}, the countable generic poset $(\mathbb{P},\le)$ is homomorphism-homogeneous. Thus, the class of finite posets has the \HAP.
        \item By \cite[Proposition 15]{CamLoc10})  the countable generic strict poset $(\mathbb{P},<)$ is homomorphism-homogeneous. Thus, the class of finite strict partial orders has the \HAP.
        \item By \cite[Proposition 25]{CamLoc10} and \cite[Theorem 4.5]{Mas07}, we have that the structure  $(\mathbb{Q},\le)$ is homomorphism-homogeneous. Thus, the class of finite linear orders has the \HAP.
        \item By \cite[Proposition 15]{CamLoc10} the structure  $(\mathbb{Q},<)$ is homomorphism-homogeneous. Thus, the class of finite strict linear orders has the \HAP.        
    \end{itemize}
    For another group of ages we observe the \HAP in a more direct way:
    \begin{itemize}
    \item It was shown in \cite[Lemma 3.5]{Dol14} that the class of finite metric spaces with rational distances has the \HAP. 
    \item The same construction as in \cite[Lemma 3.5]{Dol14} shows that the class of finite metric spaces with rational distances $\le 1$ has the \HAP.
    \item Every homomorphism between tournaments is an embedding. Thus, since the class of finite tournaments has the \AP, it follows that it also has the \HAP.
    \end{itemize}

On the other hand, a number of prominent countable homogeneous structures   fails to be homomorphism-homogeneous, and thus, their ages do not have the \HAP. This list includes the Henson graphs (cf.~\cite{Hen71}) and the Henson digraphs (cf.~\cite{Hen72}). 
\end{example}

\section{Universal homogeneous endomorphisms}
\begin{definition}
    Let $\bU,\bA$ be relational structures of the same type,  let $u$ be an endomorphism of $\bU$, and let $h\colon \bA\to\bU$ be a homomorphism. If there exists an embedding $\iota\colon \bA\injto\bU$ such that $h=u\circ\iota$, then we say that \emph{$h$  factors through $u$ by $\iota$}. 
\end{definition}

\begin{definition}
    Let $\bU$ be a countable relational structure. An endomorphism $u$ of $\bU$ is called \emph{universal} if for every $\bA\in\overline{\Age(\bU)}$ we have that every  homomorphism $h\colon \bA\to\bU$ factors through $u$ by some embeddings $\iota:\bA\injto\bU$.
\end{definition}

\begin{definition}
    Let $\bU$ be a countable relational structure. An endomorphism $u$ of $\bU$ is called \emph{homogeneous} if for every $\bA\in\Age(\bU)$, for every homomorphism $h\colon \bA\to\bU$, and for all factorization $h=u\circ\iota_1=u\circ\iota_2$ by embeddings $\iota_1,\iota_2\colon \bA\injto\bU$, there exits an automorphism $f$ of $\bU$, such that $f\circ\iota_1=\iota_2$, and such that $u\circ f=u$. 
\end{definition}

\begin{remark}
    Universal homogeneous endomorphisms were introduced in \cite{PecPec13a}, where they were mainly used for the characterization of retracts of  homogeneous structures. 
\end{remark}

\begin{definition} Let $\cC$ be a class of $\uSigma$-structures. We say that $\cC$ has the 
  \emph{amalgamated extension property} (\AEP) if for all $\bA,\bB_i, \bT\in\cC$, $f_i\colon \bA\injto\bB_i$,
  $h_i\colon \bB_i\to \bT$ (where $i\in\{1,2\}$), with $h_1\circ f_1= h_2\circ f_2$, there exist $\bC\in\cC$, $g_i\colon \bB_i\injto\bC$
  (where $i\in\{1,2\}$), $\bT'\in\cC$, $h\colon \bC \to \bT'$,
  $k\colon \bT\injto\bT'$ such that the following diagram commutes:
  \[
  \begin{tikzcd}
    ~     & & &\bT'\\
     &  & \bT\urar[hook,dashed][swap]{k}\\
    \bB_1 \arrow[bend left]{urr}{h_1} \rar[hook,dashed]{g_1}
    & \bC\arrow[bend left,dashed]{uurr}[near end]{h}\\
    \bA
    \uar[hook]{f_1}\rar[hook]{f_2} & \bB_2. \arrow[bend
    right]{uur}[swap]{h_2} \uar[dashed,swap,hook]{g_2}
  \end{tikzcd}
  \]
\end{definition}

The following is a complete characterization of all countable homogeneous structures that have a universal homogeneous endomorphism:
\begin{proposition}[{\cite[Proposition 4.7]{PecPec13a}}]\label{univhomendo}
    Let $\bU$ be a countably infinite homogeneous structure. Then $\bU$ has a universal homogeneous endomorphism if and only if $\Age(\bU)$ has the \AEP and the \HAP.
\end{proposition}

\begin{definition}
    For a class $\cC$ of $\uSigma$-structures, by $(\cC,\to)$ we will denote the category that has the elements of $\cC$ as objects and all homomorphisms between the elements of $\cC$ as morphisms. Analogously, by $(\cC,\injto)$ we will denote the subcategory of $(\cC,\to)$ whose morphisms are all embeddings between structures of $\cC$.   
\end{definition}
The following is going to be useful in order to identify relational structures whose age has the \AEP:
\begin{definition}[{\cite[Section 1.1]{Dol12}}]
	Let $\bU$ be a countably infinite $\uSigma$-structure. Then we say that $\Age(\bU)$ has the \emph{strict amalgamation property} (strict \AP) if for all $\bA,\bB_1,\bB_2\in\Age(\bU)$, and for all embeddings $f_1\colon \bA\injto\bB_1$, $f_2\colon \bA\injto\bB_2$ there exists some $\bC\in\Age(\bU)$ and embeddings $g_1\colon \bB_1\injto\bC$, $g_2\colon \bB_2\injto\bC$ such that the following is a \emph{pushout-square} in   the category $(\overline{\Age(\bU)},\to)$: 
	\[
	\begin{tikzcd}
		\bB_1 \rar[hook,dashed]{g_1}& \bC\\
		\bA  
		\uar[hook]{f_1}\rar[hook,swap]{f_2}&  \bB_2.\uar[hook,dashed]{g_2}
	\end{tikzcd}
	\] 
	That is, if $\bT\in\overline{\Age(\bU)}$, and if $h_1\colon \bB_1\to\bT$, $h_2\colon \bB_2\to\bT$ are homomorphism such that $h_1\circ f_1=h_2\circ f_2$, then there exists a unique homomorphism $h\colon \bC\to\bT$, such that the following diagram commutes:
  \[
  \begin{tikzcd}
  ~   & & \bT\\
    \bB_1 \arrow[bend left]{urr}{h_1} \rar[hook]{g_1}
    & \bC\urar[dashed]{h}\\
    \bA
    \uar[hook]{f_1}\rar[hook]{f_2} & \bB_2. \arrow[bend
    right]{uur}[swap]{h_2} \uar[swap,hook]{g_2}
  \end{tikzcd}
  \]
\end{definition}

\begin{lemma}
    Let $\bU$ be a countable $\uSigma$-structure whose age has the strict \AP. Then $\Age(\bU)$ has the \AEP, too.
\end{lemma}
\begin{proof}
    Let $\bA, \bB_1, \bB_2, \bT\in\Age(\bU)$, let $f_1\colon \bA\injto\bB_1$, $f_2\colon \bA\injto\bB_2$ be embeddings and let $h_1\colon \bB_1\to\bT$, $h_2\colon \bB_2\to\bT$, such that $h_1\circ f_1=h_2\circ f_2$. Since $\Age(\bU)$ has the strict amalgamation property, there exists a $\bC\in\Age(\bU)$,  and embeddings $g_1\colon \bB_1\injto\bC$, $g_2\colon \bB_2\injto\bC$, such that the following diagram is a pushout-square in $(\overline{\Age(\bU)},\to)$:
	\begin{equation}\label{posquare}
	\begin{tikzcd}
		\bB_1 \rar[hook,dashed]{g_1}& \bC\\
		\bA  
		\uar[hook]{f_1}\rar[hook,swap]{f_2}&  \bB_2.\uar[hook,dashed]{g_2}
	\end{tikzcd}
	\end{equation}
      Since $h_1\circ f_1=h_2\circ f_2$, and since \eqref{posquare} is a pushout-square, there exists a unique homomorphism $h\colon \bC\to\bT$ that makes the following diagram commutative:
        \[
  \begin{tikzcd}
  ~   & & \bT\\
    \bB_1 \arrow[bend left]{urr}{h_1} \rar[hook]{g_1}
    & \bC\urar[dashed]{h}\\
    \bA
    \uar[hook]{f_1}\rar[hook]{f_2} & \bB_2. \arrow[bend
    right]{uur}[swap]{h_2} \uar[swap,hook]{g_2}
  \end{tikzcd}
  \]
    Now, we can put $\bT':=\bT$, and we can define $k\colon \bT\injto\bT'$ to be the identical embedding, and we obtain that the following diagram commutes, too:
      
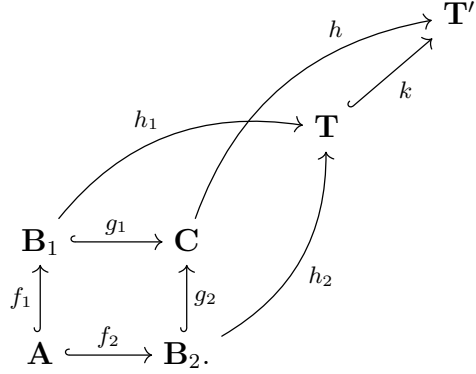
\[
  \begin{tikzcd}
    ~     & & &\bT'\\
     &  & \bT\urar[hook][swap]{k}\\
    \bB_1 \arrow[bend left]{urr}{h_1} \rar[hook]{g_1}
    & \bC\arrow[bend left]{uurr}[near end]{h}\\
    \bA
    \uar[hook]{f_1}\rar[hook]{f_2} & \bB_2. \arrow[bend
    right]{uur}[swap]{h_2} \uar[swap,hook]{g_2}
  \end{tikzcd}
  \]
  This shows, that $\Age(\bU)$ has the \AEP.  
\end{proof}

A special case of the strict amalgamation property is the free amalgamation property:
\begin{definition}
    Let $\bA$, $\bB_1$, $\bB_2$ be $\uSigma$-structures, such that $\bA\le\bB_1$, $\bA\le\bB_2$, and such that $B_1\cap B_2=A$. Then the \emph{amalgamated free sum}  of $\bB_1$ and $\bB_2$ with respect to $\bA$ is the $\uSigma$-structure $\bB_1\oplus_\bA \bB_2$ with carrier $B_1\cup B_2$, such that for each $\varrho\in\Sigma$ we have 
    \[
        \varrho^{\bB_1\oplus_\bA\bB_2} = \varrho^{\bB_1}\cup\varrho^{\bB_2}
    \]    
\end{definition}
\begin{definition}[{cf.\ \cite[Page 1602]{Mac11}}]
    An age $\cC$ of $\uSigma$-structures is sayed to have the \emph{free amalgamation property} (free \AP) if $\cC$ is closed with respect to amalgamated free sums. A homogeneous structures whose age has the free amalgamation property is called \emph{free homogeneous}. 
\end{definition}
\begin{lemma}
    Let $\bU$ be a countably infinite $\uSigma$-structure, such that $\Age(\bU)$ has the free amalgamation property. Then $\Age(\bU)$ has the strict amalgamation property, too.
\end{lemma}
\begin{proof}
    It is easy to see that the if $\bA$, $\bB_1$, and $\bB_2$ are $\uSigma$-structures with $\bA\le\bB_1$, $\bA\le\bB_2$, and $B_1\cap B_2=A$, then the following is a pushout-square in the category of all $\uSigma$-structures:
	\[
	\begin{tikzcd}
		\bB_1 \rar[hook]{=}& \bB_1\oplus_\bA\bB_2\\
		\bA  
		\uar[hook]{=}\rar[hook,swap]{=}&  \bB_2.\uar[hook]{=}
	\end{tikzcd}
	\]
	Since $(\overline{\Age(\bU)},\to)$ is a full subcategory of the category of all $\uSigma$-struc\-tures, it follows that amalgamated free sums in $\Age(\bU)$ are pushouts in $(\overline{\Age(\bU)},\to)$, too. Consequently, $\Age(\bU)$ has the strict amalgmaation property.  
\end{proof}

\begin{example}\label{exAEP}
Often it is easier to observe the strict \AP rather than the \AEP. In particular, the ages of the following relational structures  have the strict \AP, and have therefore also the \AEP:
\begin{itemize}
\item the Rado graph (because it is a free homogeneous structure),
\item the countable generic poset $(\mathbb{P},\le)$ (cf.\ \cite[Pages 7,8]{Dol14}),
\item the countable generic strict poset $(\mathbb{P},<)$ (by the same argument as for $(\mathbb{P},\le)$),
\item the Henson-graphs (because they  are free homogeneous structures, cf.\ \cite[Example 2.2.2]{Mac11}),
\item the Henson-digraphs (because they  are free homogeneous structures, cf.\ \cite[Page 1604]{Mac11}).
\end{itemize}
There is also a number of ages with the \AEP but without the strict \AP:
\begin{itemize}
    \item The class of finite tournaments has the \AP. Since every homomorphism between tournaments is an embedding, it follows that the class of all finite tournaments trivially fulfills the \AEP.
    \item For the same reason as above, the class of finite strict linear orders satisfies the \AEP.
    \item The class of finite (non-strict) linear orders satisfies the \AEP (cf.\ \cite[Proposition 3.23]{Kub13}).
    \item The class of finite metric spaces with rational distances has the \AEP (implicit in \cite{Kub13}).
    \item The class of finite metric spaces with rational distances $\le 1$ has the \AEP (implicit in \cite{Kub13}).
\end{itemize}
Using Example~\ref{exHAP} together with Proposition~\ref{univhomendo}, we obtain that the following structures have universal homogeneous endomorphisms:
\begin{itemize}
    \item the Rado graph,
    \item the countable generic tournament,
    \item the countable generic strict poset $(\mathbb{P},<)$,
    \item the countable generic poset $(\mathbb{P},\le)$,
    \item the rationals with strict order $(\mathbb{Q},<)$,
    \item the rationals with the non-strict order $(\mathbb{Q},\le)$,
    \item the rational Urysohn-space,
    \item the rational Urysohn-sphere.
\end{itemize}
\end{example}
\begin{remark}
    For some structures we can give an explicit description of a universal
homogeneous endomorphism. For the countable generic tournament and for
$(\mathbb{Q},<)$ the identical automorphism is a universal homogeneous
endomorphism. For $(\mathbb{Q},\le)$ a universal homogeneous endomorphism was
described in \cite[Remark on page 32]{PecPec13a}. In \cite{LocTru12}, a generic
endomorphism of $(\mathbb{Q},\le)$ was described. This endomorphism turns out
to be universal homogeneous in our sense. It would be interesting to examine
the relations between generic endomorphisms and universal homogeneous endomorphism. 
\end{remark}

\section{Strong gate coverings}
\begin{definition}\label{MonoidStrongGateCovering}
    Let $A$ be a countably infinite  set, let $\monoid{M}\le \monoid{T}_A$ be a
transformation monoid, let $\monoid{G}$ be the group of units in $\monoid{M}$,
and let $\overline{\monoid{G}}$ be the closure of $\monoid{G}$ in $\monoid{M}$.
Then we say that $\monoid{M}$  has a \emph{strong gate covering} if there
exists an open covering $\cU$ of $\monoid{M}$ and elements $f_U\in U$, for
every $U\in\cU$, such that for all $U\in\cU$ and for all Cauchy-sequences
$(g_n)_{n\in\bN}$ of elements from $U$  there exist Cauchy-sequences
$(\kappa_n)_{n\in\bN}$ and $(\iota_n)_{n\in\bN}$ of elements from
$\overline{\monoid{G}}$ such that for all $n\in\bN$ we have
    \[
        g_n = \kappa_n\circ f_U\circ\iota_n.
    \]
\end{definition}
\begin{remark}
    Strong gate coverings appear implicitly for the first time in
\cite{BodPinPon13}. In particular, it is shown there that the endomorphism
monoid of the Rado graph has a strong gate covering.  
\end{remark}

\begin{lemma}\label{closeiota}
  Let $\bU$ be a relational structure that has a universal homogeneous
endomorphism $u$. Let $\bA$ be a finite substructure of $\bU$. Let $f,g$ be 
   endomorphisms of $\bU$ that agree on $A$. Then
  there exist selfembeddings $\iota_1,\iota_2$, such that
  \begin{enumerate}
  \item $f=u\circ\iota_1$,
  \item $g=u\circ\iota_2$,
  \item $\iota_1\restr_A=\iota_2\restr_A$.
  \end{enumerate}
\end{lemma}
\begin{proof}
  Since $u$ is universal, there exist selfembeddings $\iota_1$ and $\iota_2$ of
$\bU$, such that 
  \begin{align*}
    f&=u\circ \iota_1, \\
    g&=u\circ \iota_2.
  \end{align*}
  Let $\hat\iota_i:=\iota_i\restr_A$, for $i\in\{1,2\}$, and let $\hat{f}:=
  f\restr_A$. Let $a\in A$. Then we compute
  \[\hat{f}(a) = f(a)=u(\iota_1(a))=u(\hat\iota_1(a)).\]
  Moreover,
  \[\hat{f}(a) = f(a)=g(a) =u(\iota_2(a))=u(\hat\iota_2(a)).\]
  Since $u$ is homogeneous, there exists an automorphism $h$ of
  $\bU$, such that $h\circ\hat\iota_1=\hat\iota_2$, and such that $u\circ h=u$. Let
  $\tilde\iota_1:=h\circ\iota_1$. Then
  $\tilde\iota_1\restr_A=h\circ\hat\iota_1=\hat\iota_2=\iota_2\restr_A$. Moreover,    we have
  \begin{align*}
    u\circ\tilde\iota_1&=u\circ h\circ\iota_1=  u\circ \iota_1= f.
  \end{align*}
\end{proof}

\begin{proposition}\label{contiota}
  Let $\bU$ be a countably infinite  relational structure that has a universal
homogeneous endomorphism $u$. Let $(f_j)_{j<\omega}$ be a sequence of 
  endomorphisms of $\bU$ that converge to an endomorphism $f$ of $\bU$. Then there is a sequence  
  $(\iota_{j})_{j<\omega}$ of homomorphic selfembeddings of $\bU$, such that
  \begin{enumerate}
  \item for every $j<\omega$ we
    have $f_j = u\circ\iota_j$,
  \item $(\iota_{j})_{j<\omega}$ converges to $\iota\in\Emb(\bU)$,
  \item $f = u\circ\iota$.
  \end{enumerate}
\end{proposition}
\begin{proof}
  Since $u$ is a universal homogeneous endomorphism of $\bU$, there exists a 
selfembedding $\iota$ of $\bU$ such that $f=u\circ\iota$.  

    Let $\ta=(a_i)_{i\in\omega}$ be any enumeration of $U$. 
  For every finite substructure $\bA$ of $\bU$ let $n_\bA$
  be the smallest element of $\omega$ such that $A\subseteq
  \{a_0,\dots,a_{n_\bA-1}\}$.
  
  Let $(\bA_i)_{i<\omega}$ be a sequence of finite
  substructures of $\bU$ such that $\bA_i\le\bA_j$ whenever $i\le j$
  and such that $\bigcup_i\bA_i = \bU$ (this exists because $\bU$ is
  countably infinite). Then the 
  sequence $(n_{\bA_i})_{i<\omega}$ is monotonous and unbounded.
    
  Since $(f_j)_{j<\omega}$ converges to $f$, we have that for every $i<\omega$
  there exists a $j_i<\omega$ such that for every $k>j_i$ we have that
  $D_{\ta}(f_k,f)>n_{\bA_i}$.  Without loss of generality we may assume
  that $j_i$ is chosen as small as possible.

  For $0\le k< j_0$, using the fact that $u$ is universal homogeneous, we choose
  $\iota_{k}$, such that 
  \[f_k= u\circ\iota_{k}.\]

  For $j_i\le k < j_{i+1}$, using Lemma~\ref{closeiota}, we chose
  $\iota_{k}$, such that 
  \[f_k= u\circ\iota_{k},\]
  and such that $\iota_{k}$ agrees with $\iota$ on $A_i$.

  It remains to observe that the
  sequence $(\iota_{j})_{j<\omega}$ converges to $\iota$. Let
  $\varepsilon>0$ and let $N:=
  \max(-\lfloor\log_2(\varepsilon)\rfloor,1)$. Then there exists an
  $i<\omega$, such that $\{a_0,\dots,a_{N-1}\}\subseteq A_i$. But
  then, by construction, for all $k\ge j_i$, we have that $\iota_{k}$ agrees with
  $\iota$ on $\{a_0,\dots,a_{N-1}\}$ --- in particular,
  $D_\ta(\iota_{k},\iota)\ge N$, and thus
  $d_\ta(\iota_{k},\iota)\le\varepsilon$.
\end{proof}

\begin{proposition}\label{monoidstronggateexistence}
    If $\bU$ is a countable relational structure that has  a  universal
homogeneous endomorphism, then $\End(\bU)$ has a strong gate covering. 
\end{proposition}
\begin{proof}
    This is a direct consequence of Propositions~\ref{contiota}, taking
$\cU=\{\End(\bU)\}$ as an open covering of $\End(\bU)$, and using that
$(\End(\bU),d_\ta)$ is a complete metric space, for each enumeration $\ta$ of $U$.    
\end{proof}

\section{Automatic homeomorphicity}

\begin{definition}
    Let $\mathcal{K}$ be a class of structures and let $\bA\in\cK$.  We say
that $\End(\bA)$ has \emph{automatic homeomorphicity with respect to
$\mathcal{K}$} if every monoid isomorphism from $\End(\bA)$ the the
endomorphism monoid of a member of  $\mathcal{K}$ is a homeomorphism.    
\end{definition}

\begin{lemma}\label{univcontmonoid}
    Let $A$, $B$ be countable sets, and let $\monoid{M}_1\le \monoid{T}_A$,
$\monoid{M}_2\le \monoid{T}_B$ be monoids, such that $\monoid{M}_1$ has a dense
set of units. Let $\ta$ and $\tb$ be enumerations of $A$ and $B$, respectively.
 
    Let $h\colon \monoid{M}_1\to \monoid{M}_2$ be a continuous homomorphism.  
Then $h$ is uniformly continuous from $(\monoid{M}_1, d_\ta)$ to
$(\monoid{M}_2,d_\tb)$.  
\end{lemma}
\begin{proof}
    Let $e_1$, $e_2$ be the neutral elements of $\monoid{M}_1$ and of
$\monoid{M}_2$, respectively. Let $\varepsilon>0$.   Since $h$ is continuous at
$e_1$, there exists a $\Delta\in\bN\setminus\{0\}$  such that, with  
$\delta:=2^{-\Delta}$, for all $m\in \monoid{M}_1$ with $d_\ta(m,e_1)\le
\delta$ we have $d_\tb(h(m),e_2)\le\varepsilon$.
    
    Let $m,m'\in \monoid{M}_1$ with  $d_\ta(m,m')\le \delta$. Then we have
    \[
    (m(a_0),\dots,m(a_{\Delta-1}))=(m'(a_0),\dots,m'(a_{\Delta-1}))=:\tc.
    \]
    But since the units lie dense in $\monoid{M}_1$, there exists a unit $g\in
\monoid{M}_1$ with  
    \[
        (g(a_0),\dots,g(a_{\Delta-1}))=\tc.
    \] 
    Consider now $\widetilde{m}:= g^{-1}m$ and $\widetilde{m}':=g^{-1}m'$. Then
$d_\ta(\widetilde{m},e_1)\le\delta$ and $d_\ta(\widetilde{m}',e_1)\le \delta$.  
    
    Now we compute
    \begin{align*}
        \varepsilon\ge
d_\tb(h(\widetilde{m}),e_2)&=d_\tb(h(g^{-1}m),e_2)=d_\tb(h(g)^{-1}h(m),e_2)\\ 
        &= d_\tb(h(m),h(g)) 
    \end{align*}
    In the same way we obtain $d_\tb(h(m'),h(g))\le\varepsilon$. Hence, since
$d_\tb$ is an ultrametric, we have $d_\tb(h(m),h(m'))\le\varepsilon$.  
\end{proof}

We will need the following basic facts about metric spaces and uniformly
continuous functions: 
\begin{lemma}\label{uniqueext}
    Let $(\metsp{M}_1,d_2)$ be a metric space and let $(\metsp{M}_2,d_2)$ be a
complete metric space. Then every uniformly continuous function $f$ from
$(\metsp{M}_1,d_1)$ to $(\metsp{M}_2,d_2)$ has a unique uniformly continuous
extension to the completion of $(\metsp{M}_1,d_1)$.    
\end{lemma}

\begin{lemma}
    Let $\Met$ be the category of metric spaces with uniformly continuous
functions. Let $\cMet$ be the full subcategory of $\Met$ spanned by all
complete metric spaces. Let $U\colon \cMet\injto\Met$ be the inclusion functor.
Then $U$ has a left-adjoint functor $C$, mapping each metric space $\metsp{M}$
to its completion $\overline{\metsp{M}}$ and every uniformly continuous
function $f\colon \metsp{M}_1\to\metsp{M}_2$ to its unique extension
$\hat{f}\colon \overline{\metsp{M}}_1 \to\overline{\metsp{M}}_2$.      
\end{lemma}
\begin{proof}
Folklore, cf.\ \cite[Page 92]{Mac98}
\end{proof}

\begin{lemma}\label{closurecompl}
    Let $A$ be a countably infinite set and let $\monoid{G}$ be a closed
subgroup of $\monoid{S}_A$.  Let  $\ta=(a_i)_{i<\omega}$ be an enumeration of $A$.  
    Then the closure of $\monoid{G}$ in $\monoid{T}_A$ coincides with the
Cauchy-completion of $\monoid{G}$ in $(\monoid{T}_A,d_\ta)$.  
\end{lemma}
\begin{proof}
    This follows immediately from the fact that $(\monoid{T}_A,d_\ta)$ is a
complete metric space, and that complete subspaces of complete metric spaces
are closed, and, vice versa, closed subspaces of complete metric spaces are complete.   
\end{proof}

\begin{proposition}\label{monoidautocontcriterion}
    Let $\bA$ and $\bB$ be two countable relational structures, such that
$\End(\bB)$ has a strong gate covering. 
    Let $h\colon \End(\bA)\to\End(\bB)$ be a continuous monoid-isomorphism.
Then $h$ is a homeomorphism. 
\end{proposition}
Before coming to the proof, let us recall the similar result by  Lascar for closed permutation groups:
\begin{proposition}[{\cite[Corollary 2.8]{Las91}}]\label{propLascar}
Let $\bA$ and $\bB$ be countable structures and let $f:\Aut(\bA)\to\Aut(\bB)$ be a continuous isomorphism. Then $f$ is a homeomorphism. 
\end{proposition}
\begin{proof}[Proof of Proposition \ref{monoidautocontcriterion}]
    Let $\ta$ and $\tb$ be enumerations of $A$ and $B$, respectively. Let 
$d_\ta$ and $d_\tb$ be the ultrametrics induced by $\ta$ and $\tb$ in
$\monoid{T}_A$ and $\monoid{T}_B$, respectively. Let further $f:=
h\restr_{\Aut(\bA)}$.    
    
    Since $h$ is continuous, we have that $f\colon \Aut(\bA)\to\Aut(\bB)$ is
continuous, too.  
    By Proposition~\ref{propLascar}, $f$ is a homeomorphism.
    Thus, by Lemma~\ref{univcontmonoid}, $f\colon
(\Aut(\bA),d_\ta)\to(\Aut(\bB),d_\tb)$  and $f^{-1}\colon
(\Aut(\bB),d_\tb)\to(\Aut(\bA),d_\ta)$ are uniformly continuous, i.e., they are
isomorphisms in the category $\Met$.    
    
    Let $\hat{f}:=C(f)$ be the unique extension of $f$ to
$\overline{\Aut(\bA)}$ as a uniformly continuous function. Then $\hat{f}$ is an
isomorphism in the category $\cMet$ and $C(f^{-1})=C(f)^{-1}$.

    Let now $g:=h\restr_{\overline{\Aut(\bA)}}$. Since $h$ is continuous, we have that $g:\overline{\Aut(\bA)}\to\overline{\Aut(\bB)}$ is continuous, too. Thus, from Lemma~\ref{univcontmonoid} we conclude that $g:(\overline{\Aut(\bA)},d_A)\to (\overline{\Aut(\bB)},d_B)$ is uniformly continuous. 
Since, clearly, we have $g\restr_{\Aut(\bA)}=f$, we conclude from Lemma~\ref{uniqueext}, that $g=C(f)=\hat{f}$. Thus $g:\overline{\Aut(\bA)}\to\overline{\Aut(\bB)}$ is a homeomorphism. 

Now we are ready to show that $h^{-1}$ is continuous:
Let $(v_n)_{n\in\bN}$ be a Cauchy-sequence of endomorphisms of $\bB$. Since $(\End(\bB),d_\tb)$ is complete, $(v_n)_{n\in\bN}$ is convergent --- say to $v\in\End(\bB)$. 
       
Let $(\cU,(f_U)_{U\in\cU})$ be a strong gate covering of $\End(\bB)$. Then there exists a $U\in\cU$ and an $n_0\in\bN$ such that for all $n\ge n_0$ we have $v_n\in U$. Without loss of generality, assume that $n_0=0$. By the definition of strong gate coverings there exist Cauchy-sequences $(\kappa_n)_{n\in\bN}$ and $(\iota_n)_{n\in\bN}$ in $\overline{\Aut(\bB)}$, such that 
$v_n = \kappa_n\circ f_U \circ\iota_n$,
for all $n\in\bN$. In particular, with \[\kappa=\lim_{n\to\infty}\kappa_n \text{ and } \iota=\lim_{n\to\infty} \iota_n,\] we have 
$v= \kappa\circ f_U\circ\iota$.
Because $g^{-1}$ is continuous, and since $g^{-1}=(h^{-1})\restr_{\overline{\Aut(\bB)}}$,  we have 
\[\lim_{n\to\infty} h^{-1}(\kappa_n)=h^{-1}(\kappa) \text{ and } \lim_{n\to\infty}h^{-1}(\iota_n) = h^{-1}(\iota).\]  

Now, since $h^{-1}$ is a monoid-isomorphism, we have 
\[ h^{-1}(v_n) = h^{-1}(\kappa_n)\circ h^{-1}(f_U) \circ h^{-1}(\iota_n).\]
Thus, since the composition of functions is continuous, we have that the sequence $(h^{-1}(v_n))_{n\in\bN}$ converges to $h^{-1}(v)$. From this, it follows that $h^{-1}$ is continuous. 
\end{proof}

In the following we are going to adapt \cite[Proposition 27]{BodPinPon13} to 
the case of transformation monoids. In order to do so, we have to make a few
preparations: 
\begin{definition}
    Let $A$ be a countably infinite set and let $\monoid{M}\le\monoid{T}_A$.
For $a,b\in A$ define $a\preceq b$ if there exists some $h\in\monoid{M}$, such
that $h(b)=a$. Let $\sim$ be the closure of $(\preceq)$ to an equivalence
relation on $A$. Then the equivalence classes of $\sim$ will be called
\emph{weak orbits} of $\monoid{M}$ on $A$.    
\end{definition}

\begin{lemma}\label{star}
    Let $A$ and $B$  be sets, let $\varrho\subseteq A^2$ be a relation, and let $f\colon A\to B$ be a function , such that $\varrho\subseteq\ker f$. Then the closure $\varrho^{\operatorname{eq}}$ of $\varrho$ to an equivalence relation is contained in $\ker f$, too.
\end{lemma}
\begin{proof}
    This follows from the fact that the operator $-^{\operatorname{eq}}$ is monotonous and idempotent. In particular we have
    \[ \varrho^{\operatorname{eq}}\subseteq(\ker f)^{\operatorname{eq}} = \ker f.\qedhere\]
\end{proof}

\begin{proposition}\label{monoidbpopen}
    Let $\bA$ and $\bB$ be structures such that $\End(\bA)$ contains all constant functions and such that $\End(\bB)$ has only finitely many weak orbits on $B$. Then every monoid-isomorphism from $\End(\bA)$ to $\End(\bB)$ is open.
\end{proposition}
\begin{proof}
    Let $h: \End(\bA)\to\End(\bB)$ be a monoid-homomorphism. Let $a,b\in A$, and let $U=\{f\in\End(\bA)\mid f(a)=b\}$. 
    For $d\in A$ denote by $c_d$ the constant endomorphism of $\bA$ that maps everything to $d$. Then we have $U=\{f\in\End(\bA)\mid c_b = f \circ c_a\}$. Since $h$ is a monoid-isomorphism, we have that $h(U)=\{g\in\End(\bB)\mid h(c_b) = g\circ h(c_a)\}$. 
    
      Note that $c_a$ and $c_b$ are left-zeros in $\End(\bA)$. Thus, since $h$ is a monoid-isomorphism,  we have that $h(c_a)$ and $h(c_b)$ are left-zeros in $\End(\bB)$. It follows that $h(c_a)$ and $h(c_b)$ are constant on weak orbits of $\End(\bB)$ on $B$. Indeed, let $x\in B$, $g\in\End(\bB)$, and let $y:=g(x)$. Then $h(c_a)=h(c_a)\circ g$. Hence $h(c_a)(x)= h(c_a)(g(x))=h(c_a)(y)$. In other words, $(\preceq)\subseteq\ker(h(c_a))$. Hence, by Lemma~\ref{star}, 
      \[(\sim)=(\preceq)^{\operatorname{eq}}\subseteq \ker(h(c_a)),\] and the claim follows. 
      
       Let  $\{o_1,\dots,o_k\}$ be a transversal of the weak  orbits of $\End(\bB)$ on $B$.  Then we have for every $g\in\End(\bB)$ that $h(c_b)=g\circ h(c_a)$ if and only if $h(c_b)(o_i)=g(h(c_a)(o_i))$, for all $i\in\{1,\dots,k\}$. In other words, with $a_i=h(c_a)(o_i)$ and $b_i=h(c_b)(o_i)$ ($i=1,\dots,k$), we have
      \[
      h(U)=\{g\in\End(\bB)\mid g(a_i)=b_i,\, i=1,\dots,k\}.
      \] 
      Thus $h(U)$ is a finite intersection of basic open sets in $\End(\bB)$. Consequently, $h(U)$ is open.
\end{proof}

\begin{remark}
    Note that for every transformation monoid $\monoid{M}\le\monoid{T}_B$ we
have that if the group $\monoid{G}$ of units in $\monoid{M}$ is oligomorphic,
then $\monoid{M}$ has only finitely many weak orbits on $B$. On the other hand,
the monoid of non-expansive selfmaps of the rational Urysohn-space has just one
weak orbit but it automorphism group is not oligoorphic. Thus we have that the
class of countable structures whose endomorphism monoid has only finitely many
weak orbits properly contains the class of $\omega$-categorical structures.    

\end{remark}

\begin{theorem}\label{reconQ}
    Let $\bB$ be a countable structure, such that $\End(\bB)$ has only finitely many weak orbits on $B$, and let $h\colon \End(\bQ,\le)\to\End(\bB)$ be a monoid-isomorphism. Then $h$ is a homeomorphism. 
\end{theorem}
\begin{proof}
    Clearly, every constant function on $\bQ$ is an endomorphism of $(\bQ,\le)$. Thus, by Proposition~\ref{monoidbpopen}, $h$ is open. 
    
    It was shown by Kubi\'s \cite[Proposition 3.23]{Kub13} that the class of finite linear orders has the \AEP. It is known (cf. \cite{Mas07,CamLoc10}) that $(\bQ,\le)$ is homomorphism-homogeneous. Thus, the class of finite linear orders has the \HAP.  Thus, by Proposition~\ref{univhomendo}, $(\bQ,\le)$ has a universal homogeneous endomorphism (this follows also from an earlier result \cite[Proposition 4.7]{PecPec13a}).
    
    Now, by Proposition~\ref{monoidstronggateexistence}, $\End(\bQ,\le)$ has a strong gate covering. Finally, by Proposition~\ref{monoidautocontcriterion}, $h$ is continuous. Altogether we have that $h$ is a homeomorphism.
\end{proof} 

\begin{corollary}
    The endomorphism monoid $\End(\bQ,\le)$ has automatic homeomorphicity with respect to the class of countable posets.
\end{corollary}
\begin{proof}
    Let $\bB=(B,\le)$ be a countable posets. Then every constant function on $B$ is an endomorphism of $\bB$ hence $\End(\bB)$ has just one weak orbit. Thus, by Theorem~\ref{reconQ}, every isomorphism from $\End(\bQ,\le)$ to $\End(\bB)$ is a homeomorphism.
\end{proof}

\begin{theorem}\label{reconU}
    Let $\bB$ be a countable structure, such that $\End(\bB)$ has only finitely many weak orbits on $B$, and let $h$ be a monoid isomorphism from the monoid of non-expansive selfmaps of the rational Urysohn space $\mathbb{U}_0$ to $\End(\bB)$. Then $h$ is a homeomorphism. 
\end{theorem}
\begin{proof}    
    Clearly, all constant functions on $\mathbb{U}_0$ are non-expansive. Thus, by Propositon~\ref{monoidbpopen},  $h$ is open. 


The class of finite metric spaces has the \AEP (cf.\ Example~\ref{exAEP}). It was shown by Dolinka in \cite[Lemma 3.5]{Dol14} that the class of finite metric spaces has the \HAP. Thus, by Proposition~\ref{univhomendo}, $\mathbb{U}_0$ has a universal homogeneous endomorphism. 
    
    By Proposition~\ref{monoidstronggateexistence}, $\End(\mathbb{U}_0)$  has a strong gate covering. Thus, by Proposition~\ref{monoidautocontcriterion}, $h$ is continuous.

    Altogether we have that $h$ is a homeomorphism.
\end{proof}

\begin{corollary}
    The monoid of non-expansive selfmaps of the rational Urysohn space has automatic homeomorphicity with respect to the class of countable metric spaces.
\end{corollary}
\begin{proof}
    Let $\mathbb{M}$ be a countable metric space. Then every constant function on $\mathbb{M}$ is a non-expansive selfmap of $\mathbb{M}$. Thus, the monoid of non-expansive selfmaps of $\mathbb{M}$ has just one weak orbit. Thus, by Theorem~\ref{reconU}, every isomorphism between $\End(\mathbb{U}_0)$ and $\End(\mathbb{M})$ is a homeomorphism.
\end{proof}

Recall that by $(\mathbb{P},\le)$ is denoted the countable universal homogeneous partially ordered set (a.k.a. countable generic poset, or countable random poset).
\begin{theorem}\label{reconP}
    Let $\bB$ be a countable structure, such that $\End(\bB)$ has only finitely many weak orbits on $B$, and let $h\colon \End(\mathbb{P},\le)\to\End(\bB)$ be a monoid-isomorphism. Then $h$ is a homeomorphism. 
\end{theorem}
\begin{proof}    
    Clearly, all constant functions are endomorphisms of $(\mathbb{P},\le)$. Thus, by Propositon~\ref{monoidbpopen},  $h$ is open. 

%

The class of finite posets has the strict \AP. Hence, it has the \AEP. It was shown by Dolinka in \cite[Example 3.4]{Dol14} that the class of finite posets has the \HAP. Thus, by Proposition~\ref{univhomendo}, $(\mathbb{P},\le)$ has a universal homogeneous endomorphism. 
    
    By Proposition~\ref{monoidstronggateexistence}, $\End(\mathbb{P},\le)$  has a strong gate covering. Thus, by Proposition~\ref{monoidautocontcriterion}, $h$ is continuous.
    
    Altogether we have that $h$ is a homeomorphism.
\end{proof}

\section{Concluding remarks}
We conclude this paper with some open problems:

In \cite{BodJun10} it was shown, that two positive existentially bi-interpretable $\omega$-categorical structures have topologically isomorphic endomorphism monoids. Moreover, if two non-contractable $\omega$-categorical structures have topologically isomorphic endomorphism monoids, then they are positive existentially bi-interpretable. 

Unfortunately, we can not use this nice result to show reconstruction up to positive existential bi-interpretability, because our approach to show automatic homeomorphicity crucially depends on Proposition~\ref{monoidautocontcriterion}. In particular, all structures considered by us are contractable. 
We ask:
\begin{problem}
    Is the rational Urysohn-space determined up to positive existential bi-interpretabiliy by its endomorphism monoid, among all countable metric spaces?
\end{problem}

\begin{problem}
    Is $(\mathbb{Q},\le)$ determined up to positive existential bi-inter\-pre\-ta\-bility by its endomorphism monoid, among all countable posets (chains)? 
\end{problem}


\begin{thebibliography}{10}

\bibitem{AhlZie86}
G.~Ahlbrandt and M.~Ziegler.
\newblock Quasi-finitely axiomatizable totally categorical theories.
\newblock {\em Ann. Pure Appl. Logic}, 30(1):63--82, 1986.
\newblock Stability in model theory (Trento, 1984).

\bibitem{BarPhD}
S.~Barbina.
\newblock {\em Automorphism groups of omega-categorical structures}.
\newblock PhD thesis, University of Leeds, 2004.

\bibitem{Bar07}
S.~Barbina.
\newblock Reconstruction of classical geometries from their automorphism group.
\newblock {\em J. Lond. Math. Soc. (2)}, 75(2):298--316, 2007.

\bibitem{BodJun10}
M.~Bodirsky and M.~Junker.
\newblock {$\aleph_0$}-categorical structures: endomorphisms and
  interpretations.
\newblock {\em Algebra Universalis}, 64(3-4):403--417, 2010.

\bibitem{BodPinPon13}
M.~Bodirsky, M.~Pinsker, and A.~Pongr{\'a}cz.
\newblock Reconstructing the topology of clones.
\newblock {\em ArXiv e-prints}, Dec. 2013.

\bibitem{BonDelDol10}
A.~Bonato, D.~Deli{{\'c}}, and I.~Dolinka.
\newblock All countable monoids embed into the monoid of the infinite random
  graph.
\newblock {\em Discrete Math.}, 310(3):373--375, 2010.

\bibitem{CamLoc10}
P.~J. Cameron and D.~C. Lockett.
\newblock {Posets, homomorphisms and homogeneity.}
\newblock {\em Discrete Math.}, 310(3):604--613, 2010.

\bibitem{CamNes06}
P.~J. Cameron and J.~Ne{\v s}et{\v r}il.
\newblock Homomorphism-homogeneous relational structures.
\newblock {\em Combin. Probab. Comput.}, 15(1-2):91--103, 2006.

\bibitem{DelDol04}
D.~Deli{\'c} and I.~Dolinka.
\newblock The endomorphism monoid of the random graph has uncountably many
  ideals.
\newblock {\em Semigroup Forum}, 69(1):75--79, 2004.

\bibitem{DixNeuTho86}
J.~D. Dixon, P.~M. Neumann, and S.~Thomas.
\newblock Subgroups of small index in infinite symmetric groups.
\newblock {\em Bull. London Math. Soc.}, 18(6):580--586, 1986.

\bibitem{Dol12}
I.~Dolinka.
\newblock A characterization of retracts in certain {F}ra{\"\i}ss{\'e} limits.
\newblock {\em Mathematical Logic Quarterly}, 58(1-2):46--54, 2012.

\bibitem{Dol14}
I.~Dolinka.
\newblock The {B}ergman property for endomorphism monoids of some
  {F}ra{\"\i}ss{\'e} limits.
\newblock {\em Forum Math.}, 26(2):357--376, 2014.

\bibitem{Fra53}
R.~Fra{\"\i}ss{\'e}.
\newblock Sur certaines relations qui g{\'e}n{\'e}ralisent l'ordre des nombres
  rationnels.
\newblock {\em C. R. Acad. Sci. Paris}, 237:540--542, 1953.

\bibitem{Hen71}
C.~W. Henson.
\newblock A family of countable homogeneous graphs.
\newblock {\em Pac. J. Math.}, 38:69--83, 1971.

\bibitem{Hen72}
C.~W. Henson.
\newblock Countable homogeneous relational structures and {$\aleph
  _{0}$}-categorical theories.
\newblock {\em J. Symbolic Logic}, 37:494--500, 1972.

\bibitem{Her98}
B.~Herwig.
\newblock Extending partial isomorphisms for the small index property of many
  {$\omega$}-categorical structures.
\newblock {\em Israel J. Math.}, 107:93--123, 1998.

\bibitem{HerLas00}
B.~Herwig and D.~Lascar.
\newblock Extending partial automorphisms and the profinite topology on free
  groups.
\newblock {\em Trans. Amer. Math. Soc.}, 352(5):1985--2021, 2000.

\bibitem{HodHodLasShe93}
W.~Hodges, I.~Hodkinson, D.~Lascar, and S.~Shelah.
\newblock The small index property for $\omega$-stable $\omega$-categorical
  structures and for the random graph.
\newblock {\em J. Lond. Math. Soc., II. Ser.}, 48(2):204--218, 1993.

\bibitem{Hru92}
E.~Hrushovski.
\newblock Extending partial isomorphisms of graphs.
\newblock {\em Combinatorica}, 12(4):411--416, 1992.

\bibitem{HydPC}
J.~Hyde.
\newblock The restriction action of
  $\operatorname{Aut}(\operatorname{Inj}(\mathbb{Q}, \le))$ on
  $\operatorname{Aut}(\mathbb{Q}, \le)$ is faithful.
\newblock in preparation.

\bibitem{KecRos07}
A.~S. Kechris and C.~Rosendal.
\newblock Turbulence, amalgamation, and generic automorphisms of homogeneous
  structures.
\newblock {\em Proceedings of the London Mathematical Society}, 94(2):302--350,
  2007.

\bibitem{Kub13}
W.~Kubi{\'s}.
\newblock Injective objects and retracts of {F}ra{\"\i}ss{\'e} limits.
\newblock {\em Forum Mathematicum}, pages ---, Jan. 2013.

\bibitem{Las91}
D.~Lascar.
\newblock {Autour de la propri{\'e}t{\'e} du petit indice. (On the small index
  property).}
\newblock {\em Proc. Lond. Math. Soc., III. Ser.}, 62(1):25--53, 1991.

\bibitem{LocTru12}
D.~C. Lockett and J.~K. Truss.
\newblock Generic endomorphisms of homogeneous structures.
\newblock In {\em Groups and model theory}, volume 576 of {\em Contemp. Math.},
  pages 217--237. Amer. Math. Soc., Providence, RI, 2012.

\bibitem{Mac98}
S.~Mac~Lane.
\newblock {\em Categories for the working mathematician}, volume~5 of {\em
  Graduate Texts in Mathematics}.
\newblock Springer-Verlag, New York, second edition, 1998.

\bibitem{Mac11}
D.~Macpherson.
\newblock A survey of homogeneous structures.
\newblock {\em Discrete Math.}, 311(15):1599--1634, 2011.

\bibitem{MalMitRus09}
V.~Maltcev, J.~D. Mitchell, and N.~Ru{\v s}kuc.
\newblock The {B}ergman property for semigroups.
\newblock {\em J. Lond. Math. Soc. (2)}, 80(1):212--232, 2009.

\bibitem{Mas07}
D.~Ma{\v s}ulovi{\'c}.
\newblock Homomorphism-homogeneous partially ordered sets.
\newblock {\em Order}, 24(4):215--226, 2007.

\bibitem{PecPec13a}
C.~Pech and M.~Pech.
\newblock Universal homomorphisms, universal structures, and the polymorphism
  clones of homogeneous structures.
\newblock {\em ArXiv e-prints}, Feb. 2013.

\bibitem{Rub94}
M.~Rubin.
\newblock On the reconstruction of {$\aleph_0$}-categorical structures from
  their automorphism groups.
\newblock {\em Proc. London Math. Soc. (3)}, 69(2):225--249, 1994.

\bibitem{Sol05}
S.~Solecki.
\newblock Extending partial isometries.
\newblock {\em Israel J. Math.}, 150:315--331, 2005.

\bibitem{Ste10}
B.~Steinberg.
\newblock A theory of transformation monoids: combinatorics and representation
  theory.
\newblock {\em Electron. J. Combin.}, 17(1):Research Paper 164, 56, 2010.

\bibitem{Tru89}
J.~K. Truss.
\newblock Infinite permutation groups. {II}. {S}ubgroups of small index.
\newblock {\em J. Algebra}, 120(2):494--515, 1989.

\bibitem{TruVarPC}
J.~K. Truss and E.~Vargas-Garc{\'\i}a.
\newblock Reconstructing the topological monoid of self-embeddings of the
  rationals.
\newblock in preparation.

\end{thebibliography}

\end{document}